\documentclass[12pt]{article}

\usepackage{amsmath}
\usepackage{amsfonts}
\usepackage{bbm}
\usepackage{amssymb}
\usepackage{graphicx}
\usepackage{caption}
\usepackage{subcaption}
\usepackage{enumitem}
\usepackage{amsthm}

\newtheorem{theorem}{Theorem}
\newtheorem{conj}{Conjecture}
\newtheorem{lemma}{Lemma}

\begin{document}

\title{A proof of a Frankl-Kupavskii conjecture on intersecting families}
\author{Agnijo Banerjee\thanks{ab2558@cam.ac.uk, Department of Pure Mathematics and Mathematical Statistics (DPMMS), University of Cambridge, Wilberforce Road, Cambridge, CB3 0WA, United Kingdom}}
\date{9 May 2023}

\maketitle

\begin{abstract}

\noindent A family $\mathcal{F} \subset \mathcal{P}(n)$ is \textit{$r$-wise $k$-intersecting} if $|A_1 \cap \dots \cap A_r| \geq k$ for any $A_1, \dots, A_r \in \mathcal{F}$. It is easily seen that if $\mathcal{F}$ is $r$-wise $k$-intersecting for $r \geq 2$, $k \geq 1$ then $|\mathcal{F}| \leq 2^{n-1}$. The problem of determining the maximal size of a family $\mathcal{F}$ that is both $r_1$-wise $k_1$-intersecting and $r_2$-wise $k_2$-intersecting was raised in 2019 by Frankl and Kupavskii \cite{frankl-kupavskii}. They proved the surprising result that, for $(r_1,k_1) = (3,1)$ and $(r_2,k_2) = (2,32)$ then this maximum is at most $2^{n-2}$, and conjectured the same holds if $k_2$ is replaced by $3$. In this paper we shall not only prove this conjecture but we shall also determine the exact maximum for $(r_1,k_1) = (3,1)$ and $(r_2,k_2) = (2,3)$ for all $n$.

\end{abstract}

\newpage

\section{Introduction}

\noindent We say that a family $\mathcal{F} \subset \mathcal{P}(n)$ is $r$-wise $k$-intersecting if any $r$ sets in $\mathcal{F}$ have common intersection of size at least $k$. (If $r$ is omitted it is assumed to be $2$, and if $k$ is omitted it is assumed to be $1$.) We define the collections of families $\mathcal{A}_k(n)$ and $\mathcal{B}_k(n)$ to consist of all $3$-wise $k$-intersecting families in $\mathcal{P}(n)$ and all $k$-intersecting families in $\mathcal{P}(n)$ respectively. We also define, for a family $\mathcal{F} \subset \mathcal{P}(n)$, the function $w(\mathcal{F}) = \frac{|\mathcal{F}|}{2^n}$.

\bigskip

\noindent Frankl and Kupavskii \cite{frankl-kupavskii} started the investigation of the function $W_{k_1,k_2}(n) = \text{max}\{w(\mathcal{F}) : \mathcal{F} \in \mathcal{A}_{k_1}(n) \cap \mathcal{B}_{k_2}(n)\}$. This is the maximum proportion of $\mathcal{P}(n)$ that can be occupied by a family that is both $3$-wise $k_1$-intersecting and ($2$-wise) $k_2$-intersecting. Equivalently, $W_{k_1,k_2}(n) = \text{max}\{w(\mathcal{F} \cap \mathcal{G}) : \mathcal{F} \in \mathcal{A}_{k_1}(n), \mathcal{G} \in \mathcal{B}_{k_2}(n)\}$. We shall concentrate on the most important case, $k_1=1, k_2=3$. Frankl and Kupavskii conjectured the following:

\begin{conj}
\label{conj:frankl-kupavskii}
If $\mathcal{F} \subset \mathcal{P}(n)$ is both $3$-wise intersecting and $3$-intersecting, then $|\mathcal{F}| \leq 2^{n-2}$. Equivalently, $W_{1,3}(n) \leq \frac{1}{4}$ for all $n$.
\end{conj}

\bigskip

\noindent If $\mathcal{F}$ is $r$-wise $k$-intersecting then so is the up-set generated by $\mathcal{F}$. Thus, we may assume throughout this paper that $\mathcal{F}$ is an up-set. Also, applying left-compressions to $\mathcal{F}$ preserves the property of being $r$-wise $k$-intersecting so we may also assume that, except where otherwise stated, $\mathcal{F}$ is left-compressed.

\bigskip

\noindent Also, for a family $\mathcal{F} \subset \mathcal{P}(n)$, and for $k \geq 1$, we denote by $\mathcal{F} \times \{0,1\}^k$ the family $\{A \subset \mathcal{P}(n+k) : A \cap [n] \in \mathcal{F} \}$. (This arises from considering subsets of $[n]$ as binary sequences of length n). We observe that $w(\mathcal{F} \times \{0,1\}^k) = w(\mathcal{F})$ and thus that $W_{k_1,k_2}(n)$ is non-decreasing for all $k_1$ and $k_2$.

\bigskip

\noindent We note that $\text{max}\{w(\mathcal{F}) : \mathcal{F} \in \mathcal{A}_1(n) \} = \frac{1}{2}$ (achieved by, for instance, letting $\mathcal{F}$ consist of all subsets of $[n]$ containing $1$), and for every fixed $k_2 \geq 3$ we have $\text{max}\{w(\mathcal{F}) : \mathcal{F} \in \mathcal{B}_{k_2}(n) \}$ tends to $\frac{1}{2}$ as $n \to \infty$. The Harris-Kleitman inequality thus tells us that

$$\liminf_{n \to \infty} W_{1,k_2}(n) \geq \frac{1}{4}$$

\bigskip

\noindent In view of this trivial inequality, the result of Frankl and Kupavskii that $W_{1,32}(n) \leq \frac{1}{4}$ for all $n$ is surprising. In this paper we shall prove considerably more, namely that $W_{1,3}(n) \leq \frac{1}{4}$ for all $n$. Moreover, we shall determine the exact value of $W_{1,3}(n)$ for all $n$, and find the unique left-compressed families $\mathcal{F}$ with $w(\mathcal{F})=W_{1,3}(n)$. We write $W(n)$ for $W_{1,3}(n)$.

\bigskip

\noindent For small values of $n$, the families of maximal size can be found by simply taking the maximal $3$-intersecting families (as in \cite{katona}), since these will also be $3$-wise intersecting. For $n=3$, this is $\{123\}$ and $W(3) = \frac{1}{8}$. Likewise, for $n=4$, we get $\{123,1234\}$ and $W(4) = \frac{1}{8}$. For $n=5$, we can take all sets of size at least $4$ and $W(5) = \frac{3}{16}$. Likewise, for $n=6$, we can take all sets whose intersection with $[5]$ has size at least $4$ and $W(6) = \frac{3}{16}$. Our main result will thus consider $n \geq 7$.

\bigskip

\noindent For odd $n \geq 7$, we define the following family:

$$\mathcal{F}_n=\{A \subset [n] : 1 \in A, |A| \geq \frac{n+3}{2} \} \cup \{A \subset [n] : 1 \not\in A, |A| \geq n-2 \}.$$

\noindent We have that $\mathcal{F}_n$ is $3$-intersecting and $3$-wise intersecting. We also have that

$$w(\mathcal{F}_n) = \frac{1}{4} + 2^{-n} \left( - \frac{1}{2} \binom{n-1}{\frac{n-1}{2}} + n \right).$$

\noindent We note that $\lim\limits_{n \to \infty} w(\mathcal{F}_n) = \frac{1}{4}$, and $w(\mathcal{F}_n) \leq \frac{1}{4}$ for all $n \geq 7$. Also, for $n \geq 11$, we have $w(\mathcal{F}_{n+2}) > w(\mathcal{F}_n)$. However, we have $w(\mathcal{F}_7) > w(\mathcal{F}_9) > w(\mathcal{F}_{11})$, and indeed $w(\mathcal{F}_7) > w(\mathcal{F}_n)$ for odd $n$, $9 \leq n \leq 71$. We shall prove the following theorem.

\begin{theorem}
\label{theorem:Main_Thm}
For $n \geq 7$, the following hold.

If $7 \leq n \leq 72$, $W(n) = w(\mathcal{F}_7)$, and the unique left-compressed family $\mathcal{F}$ attaining $w(\mathcal{F}) = W(n)$ is $\mathcal{F} = \mathcal{F}_7 \times \{0,1\}^{n-7}$.

If $n \geq 73$ is odd, $W(n) = w(\mathcal{F}_n)$, and the unique left-compressed family $\mathcal{F}$ attaining $w(\mathcal{F}) = W(n)$ is $\mathcal{F} = \mathcal{F}_n$

If $n \geq 74$ is even, $W(n) = w(\mathcal{F}_{n-1})$, and the unique left-compressed family $\mathcal{F}$ attaining $w(\mathcal{F}) = W(n)$ is $\mathcal{F} = \mathcal{F}_{n-1} \times \{0,1\}$.
\end{theorem}

\noindent Since $w(\mathcal{F}_n) \leq \frac{1}{4}$ for all $n \geq 7$, Theorem \ref{theorem:Main_Thm} completes the proof of Conjecture \ref{conj:frankl-kupavskii}.

\section{Determining the values of $W(n)$}

\bigskip

\noindent The proof of Theorem \ref{theorem:Main_Thm} is loosely inspired by the proof of the Ahlswede-Khachatrian Theorem \cite{ahlswede-khachatrian}. We assume that $\mathcal{F}$ is a left-compressed up-set. We also say that a family of subsets of $[n]$ is trivial if every element contains $1$ and almost-trivial if every element of size $\leq n-3$ contains $1$. (We use triviality in the sense of \cite{frankl-kupavskii}.) The main idea is to consider the generating set of minimal elements of the up-set $\mathcal{F}$, and attempt to transform this into a generating set for a family of subsets of $[n-1]$, without reducing the value of $w$. Lemmas \ref{Exists-n-1-pair} and \ref{Almost-Trivial} will show that this is always possible unless $\mathcal{F}$ is almost-trivial, and Lemma \ref{Max-size-for-almost-trivial} will establish the upper bound for $w(\mathcal{F})$ if it is almost-trivial.

\bigskip

\noindent For $A \neq B \subset [n]$ and $|A|=|B|=r$, we write $A \prec B$ if, for all $1 \leq i \leq |A|$, the $i$th element of $A$ is at most the $i$th element of $B$. This is equivalent to saying that $A$ can be obtained from $B$ by left-compressions so any left-compressed family containing $B$ must contain $A$.

\bigskip

\noindent For any up-set $\mathcal{F}$, we let its generating set $\mathcal{G} = \mathcal{G}(\mathcal{F})$ be the family of all minimal elements of $\mathcal{F}$, and we say that $\mathcal{G}$ generates $\mathcal{F}$. We have that $\mathcal{G}$ is an antichain and $\mathcal{F} = \{A \subset [n] : \exists B \in \mathcal{G}, B \subset A\}$. Also, for $\mathcal{F}$ left-compressed, $\mathcal{G}$ has the property that if $A \prec B$ and $B \in \mathcal{G}$ then $\exists C \in \mathcal{G}$ with $C \subset A$ (as $A \prec B$ implies $A \in \mathcal{F}$). If a generating set $\mathcal{G}$ has this property, we say it is a left-compressed generating set (and in fact it then generates a left-compressed family).

\bigskip

\noindent For $\mathcal{G}$ a left-compressed generating set for $\mathcal{F}$ and $E \in \mathcal{G}$, we define

$$\mathcal{D}(E) = \{A \subset [n] : A \cap [\max(E)] = E\}.$$

\noindent For distinct $E$, the families $\mathcal{D}(E)$ must be disjoint, since if $A \in \mathcal{D}(E) \cap \mathcal{D}(E')$ for $E \neq E' \in \mathcal{G}$, we may assume w.l.o.g. that $\max(E) \leq \max(E')$ in which case we get $E \subset E'$, contradicting the fact that $\mathcal{G}$ is an antichain. Also, since every element of $\mathcal{D}(E)$ is a superset of $E$, we have that $\mathcal{D}(E) \subset \mathcal{F}$ for all $E \in \mathcal{G}$.

\bigskip

\noindent For $A \in \mathcal{F}$, let $E \in \mathcal{G}$ such that $E \subset A$, with $|E|$ minimal. Let $E'$ consist of the first $|E|$ elements of $A$. Then $E' \prec E$ and $E' \subset A$. If $E' \not\in \mathcal{G}$ then there a set $C \in \mathcal{G}$ with $C \subset E'$ and $|C| < |E'| = |E|$. Since $C \subset E'$, we have that $C \subset A$, contradicting minimality of $|E|$. Thus $E' \in \mathcal{G}$ and in fact $A \in \mathcal{D}(E')$.

\bigskip

\noindent We have that the families $\mathcal{D}(E)$ are all subfamilies of $\mathcal{F}$ and are disjoint for distinct $E$, and that every set in $\mathcal{F}$ is in $\mathcal{D}(E)$ for some $E \in \mathcal{G}$. Thus $\mathcal{F}$ is the disjoint union of the $\mathcal{D}(E)$ for $E \in \mathcal{G}$.

\bigskip

\noindent We thus have

$$|\mathcal{F}| = \sum_{E \in \mathcal{G}} |\mathcal{D}(E)|.$$

\bigskip

\noindent For a given $E \in \mathcal{G}$, we have $|\mathcal{D}(E)|=2^{n-\max(E)}$ so we obtain

$$w(\mathcal{F}) = \sum_{E \in \mathcal{G}} 2^{-\max(E)}.$$

\bigskip

\noindent Now, we split $\mathcal{G}$ into two parts: $\mathcal{G}_0 = \{A \in \mathcal{G} : n \in A\}$ and $\mathcal{G}_1 = \mathcal{G} \backslash \mathcal{G}_0$. If $\mathcal{G}_0$ is empty, we may consider $\mathcal{G}_1$ as a generating set for a family $\mathcal{F}' \subset \mathcal{P}(n-1)$, for which $\mathcal{F} = \mathcal{F}' \times \{0,1\}$. In this case, $w(\mathcal{F}) = w(\mathcal{F'})$.

\bigskip

\noindent We thus aim to transform $\mathcal{F}$ into a family in which $\mathcal{G}_0$ is empty, and we may do so by two means. The first is to simply remove some element $A$ of $\mathcal{G}_0$, which decreases $w(\mathcal{F})$ by $2^{-n}$. The second is to replace an element $A \in \mathcal{G}_0$ with $A' = A \backslash \{n\}$, and possibly remove some other elements of $\mathcal{G}$ that are supersets of $A'$. This does not remove any elements of $\mathcal{F}$ but it does add $A \backslash \{n\}$ (which was not previously in $\mathcal{F}$ as otherwise $A$ would not be minimal) so it increases $w(\mathcal{F})$ by at least $2^{-n}$. (In fact, this has the effect of adding $A \backslash \{n\}$ to $\mathcal{F}$ but leaving it otherwise unchanged, so it increases $w(\mathcal{F})$ by precisely $2^{-n}$). We refer to this operation as shortening $A$.

\bigskip

\noindent If every element of $\mathcal{G}_0$ is either removed or shortened, the resulting family will be of the form $\mathcal{F}' \times \{0,1\}$ for some family $\mathcal{F}' \subset \mathcal{P}(n-1)$l, so we seek to remove or shorten every element of $\mathcal{G}_0$. However, shortening may cause a violation of the intersection properties (as if $A$ is in $\mathcal{G}_0$, $B$ and $C$ are in $\mathcal{G}$, and $A' = A \backslash \{n\}$, it does not follow that $|A' \cap B| \geq 3$ or that $A' \cap B \cap C \neq \phi$) so we must avoid this situation when shortening.

\bigskip

\noindent Suppose that some $A \in \mathcal{G}_0$ cannot be shortened. Then there are two cases:

\bigskip

Case 1: $\exists B \in \mathcal{G}$ such that $|A' \cap B| < 3$.

Case 2: $\exists B, C \in \mathcal{G}$ such that $A' \cap B \cap C = \phi$.

\bigskip

\noindent We consider Case 1 first. We have $A' \cap B = (A \cap B) \backslash {n}$ so the only way this can occur is if $|A \cap B| = 3$ and $n \in A \cap B$. Suppose there was some $i \not \in A \cup B$. Then we consider $A^* = A \backslash \{n\} \cup \{i\}$. Since $A^* \prec A$, we have $A^* \in \mathcal{F}$ but $|A^* \cap B| = 2$, contradicting that $\mathcal{F}$ is $3$-intersecting. Thus $A \cup B = [n]$. We will refer to a pair $(A,B)$ of sets in $\mathcal{G}_0$ with $A \cup B = [n]$ and $|A \cap B| = 3$ as a sharp pair. If $i < j < n$ and $A \cap B = \{i,j,n\}$, we say $(A,B)$ is an $(i,j)$-sharp pair.

\bigskip

\noindent We now consider Case 2. Again, we have $A' \cap B \cap C = (A \cap B \cap C) \backslash \{n\}$ so we require $A \cap B \cap C = \{n\}$. Suppose there was some $i$ in at most one of $A$, $B$, and $C$. Then either $i \not \in B$ or $i \not \in C$, and w.l.o.g. we can assume the former. We consider $B^* = B \backslash \{n\} \cup \{i\}$, and again $B^* \prec B$ so $B^* \in \mathcal{F}$, and $A \cap B^* \cap C = \phi$, contradicting that $\mathcal{F}$ is $3$-wise intersecting. Thus every $i$ appears in exactly two of $A$, $B$, and $C$, except for $i=n$ which appears in all three. We will refer to a triple $(A,B,C)$ of sets in $\mathcal{G}_0$ with every element other than $n$ appearing in exactly two of $A$, $B$, and $C$ as a sharp triple.

\bigskip

\noindent We have thus shown that $A$ can be shortened unless it is part of a sharp pair or a sharp triple, so for $w(\mathcal{F})$ maximal we may assume that every element of $\mathcal{G}_0$ is in at least one sharp pair or triple. In fact, shortening and removing some elements will preserve the intersection properties provided that in every sharp pair or triple in which at least one element is shortened, we also have at least one element removed.

\bigskip

\noindent We have the following two lemmas, which together imply that either $w(\mathcal{F}) \leq W(n-1)$ or $\mathcal{F}$ is almost-trivial. This is helpful as it is much easier to analyse the maximal size of an almost-trivial family.

\begin{lemma}
\label{Exists-n-1-pair}
If there is no $(i,n-1)$-sharp pair in $\mathcal{G}_0$ for some $i < n-1$ then $w(\mathcal{F}) \leq W(n-1)$.
\end{lemma}

\begin{proof}

Assume there is no $(i,n-1)$-sharp pair. We consider the sharp pairs and triples in $\mathcal{G}_0$. Every sharp triple contains exactly two sets containing $n-1$ and one not containing $n-1$. In a sharp pair $(A,B)$, we know that $(A,B)$ is not $(i,n-1)$-sharp so $n-1$ is not in $A \cap B$. Thus exactly one of $A$ and $B$ contains $n-1$ and one does not.

\bigskip

\noindent We can now partition $\mathcal{G}_0$ into $\mathcal{G}_+$ and $\mathcal{G}_-$, with $\mathcal{G}_+$ consisting of those sets in $\mathcal{G}_0$ containing $n-1$ and $\mathcal{G}_-$ those not containing $n-1$. Then we may shorten all elements of whichever of $\mathcal{G}_+$ and $\mathcal{G}_-$ is larger (or either if they are the same size) and remove all elements of the other, producing a generating set $\mathcal{G}'$ for a family $\mathcal{F}'$. Since every sharp pair and triple contains both an element of $\mathcal{G}_+$ and an element of $\mathcal{G}_-$, this will preserve the intersection properties regardless of which is shortened and which is removed. Since we shorten at least as many sets as we remove, we have $w(\mathcal{F}') \geq w(\mathcal{F})$. Also, since every element of $\mathcal{G}_0$ was either removed or shortened, no element of $\mathcal{G}'$ contains $n$, so $\mathcal{G}'$ also generates a family $\mathcal{F}'' \subset \mathcal{P}(n-1)$ with $\mathcal{F}' = \mathcal{F}'' \times \{0,1\}$. Thus $w(\mathcal{F}') \leq W(n-1)$ so $w(\mathcal{F}) \leq W(n-1)$.

\end{proof}

\begin{lemma}
\label{Almost-Trivial}
If there is an $(i,n-1)$-sharp pair in $\mathcal{G}_0$ for some $i < n-1$ then $\mathcal{F}$ is almost-trivial.
\end{lemma}

\begin{proof}

We have that there is an $(i,n-1)$-sharp pair, for some $i < n-1$. Let this sharp pair be $(A,B)$. Suppose there was $C \in \mathcal{F}$ with $|C| \leq n-3$ and $1 \not\in C$. Then there is also some other $j \not\in C$ with $1 < j \leq n-1$, so if $n \in C$ we can take a compression to obtain $C' = \mathcal{C}_{jn}(C)$ which does not contain $n$ (and is in $\mathcal{F}$ since it is left-compressed), so we may assume w.l.o.g. $n \not\in C$. There is still a (possibly different) $j \not\in C$ with $1 < j \leq n-1$, since $|C| \leq n-3$. We now take $A' = \mathcal{C}_{j(n-1)}(\mathcal{C}_{1i}(A))$ and $B' = \mathcal{C}_{j(n-1)}(\mathcal{C}_{1i}(B))$. Since $\mathcal{F}$ is left-compressed, both $A'$ and $B'$ are in $\mathcal{F}$. However, $A' \cap B' = \{1,j,n\}$ so $A' \cap B' \cap C' = \phi$, contradicting that $\mathcal{F}$ is $3$-wise intersecting. Thus there is no such $C$, so $\mathcal{F}$ is almost-trivial.

\end{proof}

\noindent The next lemma concerns the maximal size of an almost-trivial family.

\begin{lemma}
\label{Max-size-for-almost-trivial}
Let $\mathcal{F}$ be almost trivial.

If $n$ is odd, then $w(\mathcal{F}) \leq w(\mathcal{F}_n)$ with equality iff $\mathcal{F} = \mathcal{F}_n$.

If $n$ is even, then $w(\mathcal{F}) < w(\mathcal{F}_{n-1})$.
\end{lemma}

\begin{proof}

First, the only elements of $\mathcal{F}$ that could possibly fail to contain $1$ are those of size at least $n-2$, and there are at most $n$ of them. Thus we may remove them to form a trivial family $\mathcal{F}'$ with $w(\mathcal{F}') \geq w(\mathcal{F}) - n2^{-n}$. We may remove $1$ from each set in $\mathcal{F}'$ to form a family $\mathcal{F}''$ of subsets of $[2,n]$, with $|\mathcal{F}'|=|\mathcal{F}''|$. Since $\mathcal{F}'$ is $3$-intersecting, $\mathcal{F}''$ is $2$-intersecting.

\bigskip

\noindent As proven in \cite{katona}, if $n$ is odd, say $n = 2l+1$, then $|\mathcal{F}''| \leq 2^{n-2} - \frac{1}{2} \binom{n-1}{l}$ with equality iff $\mathcal{F}'' = [2,n]^{(\geq l+1)}$. In this case we get precisely that $w(\mathcal{F}) \leq w(\mathcal{F}_n)$, and the equality case we get is precisely $\mathcal{F}_n$.

\bigskip

\noindent Also, from \cite{katona}, if $n$ is even, say $n = 2l+2$, then $|\mathcal{F}''| \leq 2^{n-2} - \binom{n-2}{l}$. In this case we get that $w(\mathcal{F}) \leq \frac{1}{4} + 2^{-n}(-\binom{n-2}{l}+n) = w(\mathcal{F}_{n-1}) - (n-2)2^{-n}$. The equality case we obtain is $\mathcal{F} = (\mathcal{F}_{n-1} \times \{0,1\}) \backslash \{A \subset [n] : 1,n \not\in A, |A| = n-3\}$. Thus, as some elements are excluded, we have $w(\mathcal{F}) < w(\mathcal{F}_{n-1})$.

\end{proof}

\noindent We can now prove Theorem \ref{theorem:Main_Thm} by induction. For $n=7$, we know that $\mathcal{F}$ is $3$-intersecting so, again from \cite{katona}, the $3$-intersecting family of subsets of $[7]$ of maximal size is in fact $\mathcal{F}_7$, and this is unique. Since $\mathcal{F}_7$ is also $3$-wise intersecting, it is the unique maximal family so $W(7) = w(\mathcal{F}_7)$.

\bigskip

\noindent Assume Theorem \ref{theorem:Main_Thm} holds for $n-1$. By Lemmas \ref{Exists-n-1-pair} and \ref{Almost-Trivial}, either $w(\mathcal{F}) \leq W(n-1)$ or $\mathcal{F}$ is almost-trivial. If $n$ is even then, by Lemma \ref{Max-size-for-almost-trivial}, if $\mathcal{F}$ is almost-trivial then $w(\mathcal{F}) < w(\mathcal{F}_{n-1})$, so $\mathcal{F}$ is not optimal. Thus the former case must hold so $W(n) = W(n-1)$.

\bigskip

\noindent If $n$ is odd then we must have either $w(\mathcal{F}) \leq W(n-1)$ or $w(\mathcal{F}) \leq w(\mathcal{F}_n)$. Thus we have $W(n) = \max(W(n-1),w(\mathcal{F}_n))$. Since we assume Theorem \ref{theorem:Main_Thm} holds for $n-1$, we have that $W(n) = W(n-1)$ for $n \leq 71$, since $W(n-1) = w(\mathcal{F}_7) \geq w(\mathcal{F}_n)$. However, for $n \geq 73$, we instead have $W(n) = w(\mathcal{F}_n)$ and in this case the unique maximal left-compressed up-set is $\mathcal{F}_n$. Thus, by induction, we have that the value of $W(n)$ is precisely as stated in Theorem \ref{theorem:Main_Thm} for all $n \geq 7$. This suffices to prove Conjecture \ref{conj:frankl-kupavskii}.

\section{Uniqueness of left-compressed families in Theorem \ref{theorem:Main_Thm}}

\bigskip

\noindent It remains to demonstrate uniqueness of the left-compressed families $\mathcal{F}$ satisfying $w(\mathcal{F}) = W(n)$ in the cases $n \leq 72$ or $n$ even. In this case, we know that $\mathcal{F}$ cannot be almost-trivial so must not have an $(i,n-1)$-sharp pair. We can then partition $\mathcal{G}_0$ into $\mathcal{G}_+$ and $\mathcal{G}_-$ as in the proof of Lemma \ref{Exists-n-1-pair}, and may shorten one of these and remove the other. If $|\mathcal{G}_+| \neq |\mathcal{G}_-|$ then doing so will strictly increase $w$, which is impossible since $\mathcal{F}$ is maximal. Thus $|\mathcal{G}_+| = |\mathcal{G}_-|$ and so we may shorten either one of $\mathcal{G}_+$ and $\mathcal{G}_-$ and remove the other without changing the value of $w$.

\bigskip

\noindent If we shorten $|\mathcal{G}_-|$, we remove from $\mathcal{F}$ those elements of $\mathcal{G}$ containing both $n$ and $n-1$, and add elements $A \backslash \{n\}$ for all $A \in \mathcal{G}$ containing $n$ but not $n-1$. Suppose that shortening produces the family $\mathcal{F}'$ and we have $A \in \mathcal{F}'$, $B \not\in \mathcal{F}'$ with $B \prec A$. Since $\mathcal{F}$ is left-compressed, either $B$ has been removed from $\mathcal{F}$ or $A$ has been added. If $B$ was removed, since $B \prec A$, we must have $A$ contains $n$ and $n-1$ and thus was not added. If $A \in \mathcal{G}$ it would also have been removed, so $A \not\in \mathcal{G}$. Then there is an $A' \subset A \in \mathcal{F}$ with $|A'| < |A|$, and a corresponding subset $B' \subset B$ consisting of the initial $|A'|$ elements of $B$. Since $B \prec A$, we have $B' \prec A'$ so $B' \in \mathcal{F}$, contradicting $B \in \mathcal{G}$. Otherwise, if $A$ was added, we must have $A' = A \cup \{n\} \in \mathcal{G}$ and $n-1 \not\in A$. Then we have $B' = B \cup \{n\} \prec A'$ so $B' \in \mathcal{F}$. If $B' \not\in \mathcal{G}$ then there is a proper subset $B'' \subset B'$ in $\mathcal{F}$ with $|B''| = |A|$, and then $B \prec B'$ contradicting $B \not\in \mathcal{F}'$. Thus $B' \in \mathcal{G}$, in which case we must have $n-1 \not\in B'$ so $B'$ is also shortened, once again contradicting $B \not\in \mathcal{F}'$. Thus, there can be no such $A$ and $B$, so $\mathcal{F}'$ is left-compressed.

\bigskip

\noindent In this case, by the inductive hypothesis, we conclude that $\mathcal{F}'$ is an extension of the unique maximal left-compressed family for $n-1$.

\bigskip

%%%%%%%%%%%%%%%%%%%%%%%%%%%%%%%%%%%%%%%%%%%%%%%%%%%%%%%%%%% YES, THIS IS EMPTYSET

\noindent In the case $n \leq 72$, the family $\mathcal{F}'$ must be an extension of $\mathcal{F}_7$. Thus for every $A \subset [7]$ with $|A|=5$, exactly one of $A$ and $A \cup \{n\}$ is in $\mathcal{G}$, and these are the only elements of $\mathcal{G}_-$. If $n \geq 9$, for every $B \subset [7]$ with $|B| = 3$ we can choose $A_1, A_2 \subset [7]$ with $|A_1|=|A_2|=5$ and $A_1 \cap A_2 = B$. If both correspond to sets in $\mathcal{G}_-$ containing $n$, we can perform a left-compression to one of them to replace the $n$ by $n-1$, to get two sets in $\mathcal{F}$ whose intersection is $B$. Then, for any $C \in \mathcal{F}$, if $|C \cap [7]| \leq 4$, then we can choose $B \subset [7]$ with $|B| = 3$ and $B \cap C = \phi$, contradicting that $\mathcal{F}$ is $3$-wise intersecting. Thus, for all $C \in \mathcal{F}$, we have $|C \cap [7]| \geq 4$ so we have $\mathcal{F} \subset \mathcal{F}_7 \times \{0,1\}^{n-7}$. Thus this is the unique left-compressed family $\mathcal{F}$ with $w(\mathcal{F}) = W(n)$. For $n=8$, again from \cite{katona} we have that this is the unique left-compressed extremal family for $3$-intersection, and this is also $3$-wise intersecting and thus is again the unique left-compressed family $\mathcal{F}$ with $w(\mathcal{F}) = W(n)$.

\bigskip

\noindent The final remaining case is $n \geq 74$ even. Again, as before, $\mathcal{F}'$ must be an extension of the unique maximal left-compressed family for $n-1$. By the inductive hypothesis, $\mathcal{F}'$ is an extension of $\mathcal{F}_{n-1}$. Thus, for $A \subset [n-1]$ with $|A| = \frac{n+2}{2}$ and $1 \in A$, either $A$ or $A \cup \{n\}$ is in $\mathcal{G}$, and if $A \cup \{n\}$ is in $\mathcal{G}$ it is in $\mathcal{G}_-$. In particular, if $n-1 \in A$ then $A \cup \{n\}$ cannot be in $\mathcal{G}_-$ so $A$ must be in $\mathcal{G}$. Now suppose there was some $B \in \mathcal{G}_+$ with $|B| \leq \frac{n+2}{2}$. Then we can choose $A \subset [n-1]$ with $|A| = \frac{n+2}{2}$ and $1 \in A$ such that $|A \cap B| = 2$ (if $1 \in B$ we must have $A \cap B = \{1,n-1\}$, otherwise it contains $n-1$ and some other element). Thus every element $B$ of $\mathcal{G}_+$ has $|B| \geq \frac{n+4}{2}$. If $1 \in B$ then $B \backslash \{n\}$ contains $1$ and has size at least $\frac{n+2}{2}$ so it is a superset of a set in $\mathcal{G}$, contradicting $B \in \mathcal{G}$. Thus, for all $B \in \mathcal{G}_+$, we must also have that $1 \not\in B$.

\bigskip

\noindent However, we can take two different subsets of $[n-1]$ of size $\frac{n+2}{2}$ containing $1$ and $n-1$, which are thus both in $\mathcal{G}$, and whose intersection is precisely $\{1,n-2,n-1\}$, so every set in $\mathcal{F}$ must intersect this set. For $B \in \mathcal{G}_+$, if $|B| \leq n-3$, we can perform left-compressions to obtain some $B' \in \mathcal{F}$ with $1,n-2,n-1 \not\in B'$, thus contradicting that $\mathcal{F}$ is $3$-wise intersecting. Hence, for all $B \in \mathcal{G}_+$, we have $|B| \geq n-2$. But we also have in $\mathcal{F}'$ all sets that do not contain $1$ and have size $n-3$, and all such $B$ are supersets of some such set. However, the only set that could have arisen in this manner from shortening a set in $\mathcal{G}_-$ is $[2,n-2]$ (since sets in $\mathcal{G}_-$ must exclude $n-1$) and so any $B$ may only be a superset of this set and no other set of size $n-3$ not containing $1$. This is impossible as all such $B$ contain $n-1$. Thus there are no such $B$ so $\mathcal{G}_+$ is empty, and since $|\mathcal{G}_+| = |\mathcal{G}_-|$ we have $\mathcal{G}_-$ is also empty and so $\mathcal{F} = \mathcal{F}'$, completing the proof of Theorem \ref{theorem:Main_Thm}.

\section{Concluding remarks and open problems}

In \cite{frankl-kupavskii}, Frankl and Kupavskii define monotone properties $\Pi_1$ and $\Pi_2$ to be incompatible if

$$
\liminf_{n \to \infty} \left(\max_{\mathcal{F}_1 \subset \mathcal{P}(n) \textrm{ has } \Pi_1} w(\mathcal{F}_1)\right)\left(\max_{\mathcal{F}_2 \subset \mathcal{P}(n) \textrm{ has } \Pi_2} w(\mathcal{F}_2)\right)
$$

$$
=\liminf_{n \to \infty} \left(\max_{\mathcal{F} \subset \mathcal{P}(n) \textrm{ has } \Pi_1,\Pi_2} w(\mathcal{F})\right)
$$

\bigskip

\noindent By the Harris-Kleitman inequality, the latter is always at least the former, so $\Pi_1$ and $\Pi_2$ are incompatible if equality holds in the limit as $n \to \infty$ in the Harris-Kleitman inequality.

\bigskip

\noindent Conjecture \ref{conj:frankl-kupavskii} is equivalent to the properties of being $3$-wise intersecting and $3$-intersecting being incompatible. In \cite{frankl-kupavskii}, Frankl and Kupavskii further conjecture the following.

\begin{conj}
\label{conj:frankl-kupavskii-2}
For every integer $s \geq 1$ there exists an integer $t_0(s)$ such that for all $t \geq t_0(s)$, the properties of being $3$-wise $s$-intersecting and $t$-intersecting are incompatible.
\end{conj}

\noindent The $s=1$ case of this conjecture was proved by Frankl and Kupavskii \cite{frankl-kupavskii}, and Theorem \ref{theorem:Main_Thm} implies that $t_0(1) = 3$. However, this conjecture remains open for all other values of $s$.

\bigskip

\noindent We may attempt to analyse the behaviour of $W_{k_1,k_2}(n)$ for $k_2 > k_1 > 1$ to attempt to prove this conjecture for $k_1 = s > 1$. However, the methods in this paper are not sufficient to do so, as the proof of the analogue of Lemma \ref{Exists-n-1-pair} fails for $k_1 > 1$. In this case, the analogue of a sharp triple is a triple $A,B,C$ of sets in $\mathcal{G}_0$ in which $|A \cap B \cap C| = k_1$ and all elements of $[n]$ are contained in at least two of $A$, $B$, and $C$. Unlike in the case $k_1=1$, if $k_1 > 1$ it is possible for all three of these to contain $n-1$, in which case it is not possible to shorten $\mathcal{G}_+$ and remove $\mathcal{G}_-$. Likewise, if $k_2 > k_1 + 2$, the proof of the analogue of Lemma \ref{Almost-Trivial} also fails.

\section*{Acknowledgement}

The author would like to thank Professor B\'{e}la Bollob\'{a}s for his valuable input.

\bigskip

\noindent The author is supported by EPSRC (Engineering and Physical Sciences Research Council).

\end{document}